\documentclass[10pt]{article}
\usepackage{amssymb,amsfonts}
\usepackage{pb-diagram}
\usepackage{amsmath,amsthm,amsfonts,amssymb}
\usepackage[usenames]{color}
\usepackage{hyperref}
\usepackage{soul}
\usepackage{graphicx}
\usepackage{amssymb}
\usepackage{amsmath}
\usepackage{amssymb}
\usepackage{mathrsfs}

\newtheorem{theorem}{Theorem}
\newtheorem{definition}{Definition}

\newtheorem{lemma}{Lemma}
\newtheorem{cor}{Corollary}
\newtheorem{remark}{Remark}

\newtheorem{prop}{Proposition}

\newcommand{\be}{\begin{enumerate}}
\newcommand{\ee}{\end{enumerate}}
\newcommand{\beq}{\begin{equation}}
\newcommand{\eeq}{\end{equation}}

\def\N{{\mathbb{N}}}
\def\Z{{\mathbb{Z}}}

\def\MB{{\mathbb{B}}}

{}
{}

\title{Bi-interpretability of  Some Monoids with the Arithmetic and Applications}
\author{Olga Kharlampovich\footnote{Hunter College, CUNY}\hspace{.1cm}  and Laura L\'opez\footnote{Graduate Center, CUNY}}

\date{}

\begin{document}

\maketitle

\begin{abstract} We will  prove bi-interpretability of the arithmetic $\N = \langle N, +,\cdot, 0, 1\rangle$ and the 
weak second order theory of $\N$ with the free monoid $\mathbb{M}_X$ of finite rank greater than 1 and with a non-trivial partially commutative monoid with trivial center. This bi-interpretability implies that finitely generated submonoids of these monoids are definable.  Moreover, any recursively enumerable language in the alphabet $X$ is definable in $\mathbb{M}_X$.   Primitive elements, and, therefore, free bases are  definable in the free monoid.   It has the so-called QFA property, namely there is a sentence $\phi$ such that every finitely generated monoid satisfying $\phi$ is isomorphic to $\mathbb{M}_X$. The same is true for a partially commutative monoid without center. We also prove that there is no  quantifier elimination in the theory of  any structure that is bi-interpretable with $\mathbb N$ to any boolean combination of formulas from $\Pi_n$ or $\Sigma_n$.

\end{abstract}
Let $\mathbb{B} = \langle B ; \mathcal{L}(\mathbb{B})\rangle$ be an algebraic structure. A subset $A \subseteq B^n$ is called {\em definable} in $\mathbb{B}$ if there is a formula $\phi(x_1, \ldots,x_n)$ in $\mathcal{L}(\mathbb{B})$ such that  $A = \{(b_1,\ldots,b_n) \in B^n \mid \mathbb{B} \models \phi(b_1, \ldots,b_n)\}$.

In model theory an algebraic  structure $\mathbb{A} = \langle A ;f, \ldots, P, \ldots, c, \ldots\rangle$ , where $f,P,c$ stand for functions, predicates and constants,  is said to be interpretable  in a structure $\mathbb{B}$  if there is a  subset $A^* \subseteq B^n$  definable in $\mathbb{B}$, an equivalence relation $\sim$ on $A^*$ definable in $\mathbb{B}$, operations  $f^*, \ldots, $ predicates $P^*, \ldots, $ and constants $c^*, \ldots, $ on the quotient set $A^*/{\sim}$ all definable in $\mathbb{B}$ such that the structure $\mathbb{A}^* = \langle A^*/{\sim}; f^*, \ldots, P^*, \ldots,c^*, \ldots, \rangle$ is isomorphic to $\mathbb{A}$. An interpretation of $\mathbb{A}$ in a class of structures $\mathcal{C}$ is    {\em uniform}  if the formulas that interpret   $\mathbb{A}$ in a structure $\mathbb{B}$ from $\mathcal{C}$   are the same for every structure $\mathbb{B}$ from $\mathbb{C}$.

The following is a principal result on intepretability.

\begin{lemma} \label{lemma:int}
If $\mathbb{A}$ is interpretable in $\mathbb{B}$ with parameters $P$,
then for every formula $\psi(\bar{x})$ of $\mathcal{L}(\mathbb{A})$, one can
effectively construct a formula $\psi^{*}(\bar{z},P)$ of
$\mathcal{L}(\mathbb{B})$ such that for any $\bar{a} \in \mathbb{A}$, one has
that $\mathbb{A} \models \psi(\bar{a})$ if and only if $\mathbb{B}
\models [\psi(\bar{a})]^*$.

\end{lemma}   

\begin{definition}\label{def:1}
Two algebraic structures $\mathbb{A}$ and $\mathbb{B}$ are said to be
bi-interpretable if they satisfy the following conditions:

\begin{itemize}
\item $\mathbb{B}$ is interpretable in $\mathbb{A}$ as $\mathbb{B}^*$,
  $\mathbb{A}$ is interpretable in $\mathbb{B}$ as $\mathbb{A}^*$,
  which by transitivity implies that $\mathbb{A}$ is interpretable in
  $\mathbb{A}$ as $\mathbb{A}^{**}$ and $\mathbb{B}$ in $\mathbb{B}$
  as $\mathbb{B}^{**}$.
\item There is an isomorphism $\mathbb{A} \to \mathbb{A}^{**}$ which
  is definable in $\mathbb{A}$ and an isomorphism $\mathbb{B} \to
  \mathbb{B}^{**}$ definable in $\mathbb{B}$.
\end{itemize}

\end{definition}  

It follows from Quine's paper \cite{Quine}, Section 4, that for a free monoid  of  rank $n \geq 2$ 
and generating set $X=\{x_1,\ldots,x_n\}$, the arithmetic  $\langle N, +,\cdot, \uparrow , 0, 1\rangle$ is bi-interpretable in $\mathbb{M}_X$ with parameters $X$, where $x\uparrow y$ means $x^y$.  Since the predicate $z=x^y$ is computable and therefore definable in terms of addition and multiplication (see, for example, \cite{M}) it can be removed from the signature. Quine was working with the structure of concatenation with parameters, $\mathscr{C}= \langle C, \frown \rangle$, where $C$ is the set of all finite strings in a
   finite alphabet and $\frown$ is the concatenation operation. 
   This structure $\mathscr{C}$ and the free semigroup are equivalent structures. Quine did not state the monoid version of his results but the presence of  the identity doesn't make a difference and   his results  are also valid for $\mathbb{M}_X$.

  We will show that $\N$ and the weak second order theory of $\N$  are  bi-interpretable with $\mathbb{M}_X$ (with parameters $X$) using the technique that allows also to prove this for a non-trivial free partially commutative monoid ${A}_{\Gamma}$ with trivial center. Here we say that the weak second order theory of  a structure $\mathbb{B}$ is  interpretable in $\mathbb{A}$ if the first-order structure $S(\mathbb{B},\N)$ (see Section \ref{se:2} for precise definition) which has the same expressive power, is interpretable in $\mathbb{A}$.  
  
  This bi-interpretability has  interesting applications.  For example, Theorem \ref{th:sub} states that for any $k \in \mathbb{N}$, there is a formula $\psi(y,y_1,\ldots, y_k,X)$
such that $\psi(g,g_1,\ldots ,g_k,X)$ holds in $\mathbb{M}_X$ if and only if
$g$ belongs to the submonoid generated by  $g_1,\ldots,g_k$. In other
words, finitely generated submonoids of $\mathbb{M}_X$ are definable.
In contrast to this,  it was proved in \cite{KM1} and later in
\cite{pS}  that proper subgroups of a free group are not definable
(except cyclic subgroups when the language contains constants). This
was a solution of an old problem posed by Malcev.  Primitive elements, and, therefore, free bases are not definable in a free group of rank greater than 2 \cite{KM1}, but they are easily definable in the free monoid. This implies that the free monoid  is homogeneous  (two tuples realize the same types if and only if they are automorphically equivalent).
 We will show that there is no  quantifier elimination in the theory of  any structure that is bi-interpretable with $\mathbb N$  to any boolean combination of formulas from $\Pi_n$ or $\Sigma_n$ and, in particular, in  the theory of $\mathbb{M}_X$.  In contrast to this, the theory of a  free group has quantifier elimination to boolean combinations of $\forall\exists$-formulas \cite{KMel},\cite{sela}.  

We also prove that $\mathbb{M}_X$ has the so-called QFA property, namely there is a sentence $\phi$ such that every finitely generated monoid satisfying $\phi$ is isomorphic to $\mathbb{M}_X$ (in contrast to this two non-abelian free groups of different ranks are elementarily equivalent \cite{KMel}, \cite{sela}).

\section{Interpretation of $\mathbb{N}$ in some classes of monoids}

 We will
   use Quine's method from \cite{Quine}, Section 3, to interpret  $\mathbb{N}=\langle N,+,\cdot, 0,1 \rangle $ in
   $\mathbb{M}_X$ and some other monoids with
parameters.  
$\N$     can be interpreted as the centralizer of
     $x_1$ in $\mathbb{M}_{X}$. That is, the interpretation of $\mathbb{N}$ in $\mathbb{M}_{X}$ is the set
 $C(x_1)=\{x_1^n \mid  n \in \mathbb{N}\}$ and is defined by the formula $\theta (y,x_1): x_1y=yx_1 $.  

In the following lemma we will prove that $\N$ is interpretable with parameters not only in $\mathbb{M}_X$ but in a much larger class of monoids. Let $G$ be a monoid containing elements $x_1,x_2$ and  $S$ be the set of non-trivial elements of $G$ that can be represented by subwords of words in the set $$\bar S=\{ 
x_1^{i_1}x_2^{j_1}\cdots x_1^{i_k}x_2^{j_k}\ |
i_1,\ldots , i_k\in \N-\{0\}, j_1,\ldots ,j_k\in\{1,2\}\}.$$
\begin{lemma} \label{lm:1} 
$\mathbb{N}$ is interpretable in any monoid $G$ that contains two elements $x_1$ and $x_2$ such that \begin{enumerate} \item   $x_1$ generates a free cyclic submonoid  $\langle x_1\rangle$ which is  definable;  \item $S$ is definable; \item Distinct words in $\bar S$ represent distinct elements. Let $\bar w$ be the word representing $w\in S$. If $v=u_1uu_2$ and $v,u\in S$, then  $u_1,u_2\in S$ or empty and the equality $\bar v=\bar u_1\bar u\bar u_2$ is graphical. \end{enumerate}
\end{lemma}

The third assumption implies that  both $x_1$ and $x_2$ are not divisors of 1, in particular, not invertible. 

\begin{proof} $\mathbb{N}$ will be interpreted with parameter $x_1$ as $\langle x_1\rangle$.
One has to show  that the operations $+$ and
$\cdot$ and the constants $0$ and $1$ are intepretable.
To interpret addition in $\mathbb{M}_{X}$, we interpret the addition
relation $\{(m,n,k) \mid m+n=k\}$ as the set of triples of the form $(x_1^n, x_1^m, x_1^{n+m})$
which can be defined by the formula $\phi (x,y,z)$: $xy=z$. Thus, we
interpret the constant $0$ in $\mathbb{N} $ as the empty word $\emptyset$ in
$\mathbb{M}_{X}$, and it is easy to see $\emptyset$ is an identity element of
the addition operation defined by $\phi$. To interpret multiplication of $\mathbb{N}$ in $\mathbb{M}_{X}$,
 we show that set $\{(x_1^n, x_1^m, x_1^{nm})\}$ is definable.   Given $x_1^n,x_1^m$, define $w$ as

\begin{equation} 
w(x_1^n,x_1^m)=x_2^2x_1^{n+1}x_2x_1^{m+1}x_2^2x_1^nx_2x_1^{2m+1}x_2^2\ldots  x_2^{2}x_1^2x_2x_1^{mn+1}x_2^2
\end{equation} 


The element $w\in S$ is completely determined by the following conditions: 

\begin{enumerate}
\item [1)] (head)$w=x_2^2x_1^{n+1}x_2x_1^{m+1}x_2^2w_0,$
\item [2)](recursion)If $w=w_1x_2^2w_2x_2^2w_3$ and 
$w_2=v_1x_2v_2, v_1,v_2\in \langle x_1\rangle,$ and $v_1 \neq x_1$ and
$v_1 \neq x_1^2$, then $w_3=v_3x_2v_4x_2^2w_4$, where $v_1=v_3x_1, v_4=v_2x_1^m$,
\item [3)] (tail) $w=w_4x_2^2x_1^2x_2v_5x_2^2$  where $v_5 \in
  \langle x_1\rangle$ or $w=x_2^2x_1x_2x_1^{m+1}x_2^2$.


\end{enumerate}

Conditions 1)--3) can be written in $\mathcal{L}_{\{x_1,x_2\}}$. Note that in
condition 3), we take into account the case where $n=0$,
i.e. $x_1^n=\emptyset$. Let $\psi(x,y,w)$ be the formula defining $w(x,y)$, where $x,y\in \langle x_1\rangle$,
then for $x=x_1^n,y=x_1^m,$ we have $z=x_1^{nm}$ if and only if 
$$\phi (x,y,z): \,x,y,z\in \langle x_1\rangle \wedge (\exists w \, \psi (x,y,w)\wedge
(\exists w_4 \, w=w_4x_2^2x_1^2x_2zx_1x_2^2 \vee z=\emptyset)).$$

The identity element $1$ in $\mathbb{N}$ can be interpreted as
$x_1$. One can easily check that it is consistent with the interpretation of multiplication.

\end{proof}

This result implies that $\N$ is interpretable with parameters  not only in a free monoid but in many interesting monoids.  We need some definitions.

 {\em A Baumslag-Solitar monoid} is given by a presentation $\langle a,b| ab^k=b^ma\rangle.$ 

{\em Free partially commutative monoids} (also known as {\em trace monoids} or {\em right angled Artin monoids}) defined as follows.  Given a finite graph $\Gamma$ with the set of vertices  $V$ and edges $E$ we define such a monoid $A_{\Gamma}$ by generators $V$ and relations $v_1v_2=v_2v_1$ for each pair of vertices $(v_1,v_2)\in E$.
\begin{theorem} \label{th:!} $\mathbb{N}$ is interpretable (with parameters) in $\mathbb {M}_X$  and in the following classes of monoids:
\begin{enumerate}
\item [a)] Baumslag-Solitar monoids with $k,m>2$  (we do not need parameters for them); 
\item [b)] Non-commutative free partially commutative monoids;   \item [c)] One-relator monoids $G=\langle a,b,C|x=y\rangle$, where $C$ is a non-empty alphabet, some letter of $C$ appears in $y$ and neither $x$ nor $y$ end with $a$ (or one could consider the dual case).
\end{enumerate}

\end{theorem}

\begin{proof} a) We take $x_1=a, x_2=b$.  The set $S$ can be defined using the fact that both $x_1$ and $x_2$ are irreducible elements. Notice that  elements $a$ and $b$ are definable, therefore we have interpretability without parameters.

b) We take $x_1=v_1, x_2=v_2$
such that $\Gamma$ does not have an edge between $v_1$ and $v_2.$ A non-trivial  element  is irreducible if it is not a product of two non-trivial elements.  The cyclic submonoid $\langle x_1\rangle$ is definable as the set consisting of the trivial element and non-trivial elements having only $x_1$ as their irreducible divisor. The set $S$ can be defined using the fact that both $x_1$ and $x_2$ are irreducible elements. The generating set of $A_{\Gamma}$ is definable as a set of irreducible elements. The submonoid generated by $x_1,x_2$ is free and an element in $S$ cannot be represented as a word not in $\bar S$, therefore the third assumption  is also satisfied. 

Moreover, the free submonoid generated by $x_1,x_2$ is definable,  hence the statement also follows from transitivity of interpretations.
 
 c) In this case we can assume that the monoid is not free because for a free monoid all the assumptions of Lemma \ref{lm:1} are satisfied.  The element $a$ is irreducible.  By the Freiheitssatz for one relator monoids, $a, b$ generate a free submonoid since a letter from $C$ appears in the relation.  
Note that since neither $x$ nor $y$ end in an $a$, we cannot create an $a$ at the end of a word by applying relations and if $ua, va$ are equal in $G$ them $u$ and $v$ are equal in $G$ (since the derivation from $ua$ to $va$ will never touch the final $a$). 

We show by induction on word length that the centralizer of $a$ is $\langle a\rangle.$
 If $w$ commutes with $a$ then from $aw=wa$ in $G$ and by the above, we have that $w$ graphically ends in $a$, say $w=ua$.  Then $aua=uaa$ in $G$. We deduce by right cancelling $a$, as discussed above, that $au=ua$ in $G$.  By induction $u$ is a power of $a$ and hence $w$ is a power of $a$. 
 
 If $x$  contains some letters from $C$, then the submonoid  generated by $a, b$  is definable. If only $y$ contains some $c\in C$ but $b$ is contained in $x$ and $y$, then we can interchange $c$ and $b$. Suppose that $x$ only  contains $b$ and maybe $a$ and $y$ only contains $c$ and maybe $a$. Then $x\neq b, y\neq c$ because the monoid is not free. Therefore, $b,c$ are irreducible. In this case  we can replace $b$ by $bcb$ in the interpretation and in the definition of the set $S$.
\end{proof}
\begin{cor}  If $G$ is a monoid from Theorem \ref{th:!}, then the first-order theory $Th(G)$ is undecidable.
\end{cor}

To eliminate parameters from the interpretation of $\mathbb{N}$ in $\mathbb{M}_X$ given in Lemma \ref{lm:1} we need the following lemma.
 
\begin{lemma}\label{1}
The relation $\{(x_2^s,x_1^s) \mid s \in \mathbb{N}\}$ is definable in
$\mathbb{M}_{X}$ with parameters $x_1, x_2$.
\end{lemma}

\begin{proof}
We begin by defining the set $\{x_2^sx_1^s \mid s \in \mathbb{N} \}$.
Let $a=x_1x_2x_1x_2^2$. 

The monomial $f=ax_2x_1a^2x_2^2x_1^2a^3 \cdots x_2^sx_1^sa^{s+1}$ satisfies the following conditions:

\begin{enumerate}
\item [1)] $f=ax_2x_1a^2g_3 $ where $g_3$ does not begin with $a$
\item [2)] If $f=g_1\bar a g_2\bar a ag_3$ where $\bar a\in C(a),$ $g_1$ does not end in $a$, $g_2$ does not start or end with $a$ and $g_3$ does not start with $a$, then $g_3=x_2g_2x_1\bar a a^2g_4$ where $g_4$ does not start with $a$ or $g_3=\emptyset$.
\item [3)] $f=g_1\bar a u\bar a a$, where $a\in C(a),$ $g_1$ does not end in $a$ and $u$
  does not begin or end with $a$ and is not divisible by $a$.
\end{enumerate}

Conditions 1)--3) can be written in $\mathcal{L}_{\{x_1,x_2\}}$ and
uniquely define a word $f$, so let $\phi(x)$ be the formula defining
all such words $f$, then the formula $$\psi(x): \exists f,g_1,g_1',b \,
\, (\phi(f) \wedge f=g_1bxab \wedge g_1 \neq g_1'a \wedge x \neq x_1a
\wedge x \neq ax_2 \wedge b \in C(a))$$ defines the set $\{x_2^sx_1^s \mid s \in \mathbb{N} \}$.

Now the following formula defines the set of pairs of the form
$( x_2^s,x_1^s) $:

\begin{equation}
Trans(x,y): \exists z \, \psi(z) \wedge z=xy \wedge x \in C(x_2)
\wedge y \in C(x_1)
\end{equation}

\end{proof}

 The basis $X$ consists of all irreducible elements and therefore is definable in $\mathbb{M}_X$ by the formula $\theta(x): \,
 \forall y \forall z \, (x=yz \implies y=\emptyset \vee z=\emptyset$). 

\begin{theorem} \label{th:1} $\N$ is $\emptyset$-interpretable in $\mathbb{M}_X$. \end{theorem}
\begin{proof} We will give a proof for $\mathbb{M}_X$.
 Denote by $\N_{x_{i}}$ the interpretation of $\N$ as $C(x_{i})$, where $x_i$ is an element of the basis. The number $m$ is interpreted as a pair $(x_{i}^m, x_i)$. Lemma  \ref{lm:1} implies that
there is a definable isomorphism  between $\N _{x_i}$ and $\N _{x_j}$ for any
two  elements $x_i$ and $x_j$ of the basis.   
Indeed,  we can
define pairs $(x_i^s,x_j^s), i \neq j$ where $ 1 \leq i,j \leq |X|$, without
parameters with the following formula:

\begin{equation}
\phi(x,y): \exists z_1,z_2 (z_1, z_2 \in X \wedge z_1 \neq z_2 \wedge Trans'(x,y))
\end{equation}
where $Trans'(x,y)$ is the formula $Trans(x,y)$ with any occurrences of
$x_1$ and $x_2$ replaced by $z_1$ and $z_2$, respectively. 

Thus we have a definable (without parameters) equivalence relation on the set of  pairs $(x_{i}^{m},x_i)$ and factoring over this equivalence relation
we identify all the structures $\N_{x_i}$ into one structure  isomorphic to $\mathbb{N}=\langle N,+,\cdot, 0,1 \rangle $. Therefore $\mathbb{N}=\langle N,+,\cdot, 0,1 \rangle $ is $\emptyset$-interpretable in $\mathbb{M}_X$.
\end{proof}

\begin{remark} We can similarly prove $\emptyset$-interpretability of   $\mathbb{N}=\langle N,+,\cdot, 0,1 \rangle $ in the non-commutative monoid $A_{\Gamma}$. \end{remark}

\section{Bi-interpretability of $\mathbb{M}_X$ and some other monoids with 
  $S(\mathbb{N},\mathbb{N})$ }
\subsection{Interpretation of $S(\mathbb{N},\mathbb{N})$ in
  $\mathbb{M}_X$  and some other monoids}\label{se:2}

Let $\mathbb{B}$ be an algebraic structure. The three-sorted structure $S(\mathbb{B},\mathbb{N})$, termed the list
superstructure over $\mathbb{B}$, is defined as 

\begin{equation}
S(\mathbb{B},\mathbb{N})=\langle \mathbb{B}, S(B), \mathbb{N}, t(s,i,a), l(s),
\frown, \in \rangle,
\end{equation}
where $\mathbb{N}= \langle N;+, \cdot, 0, 1 \rangle$ is the standard
arithmetic, $S(B)$ is the set of all finite sequences (tuples) of
elements of $B$, $l:S(B) \to N$ is the length function, i.e., $l(s)$
is the length $n$ of a sequence $s=(s_1,\ldots,s_n)\in S(B)$ and
$t(x,y,z)$ is a predicate on $S(B) \times N \times B$ such that
$t(s,i,a)$ holds in $S(\mathbb{B},\mathbb{N})$ if and only if $s=(s_1,\ldots,s_n)
\in S(B), i \in N, 1 \leq i \leq n$, and $a=s_i \in B$. This structure
has the same expressive power as the weak second order logic of
$\mathbb{B}$.
 
  The following result is known, and it is based on two facts: the first
one is that there are effective codings of the set of all tuples of
natural numbers such that the natural operations over the tuples are
computable on their codes; and the second one is that all computably
enumerable predicates over natural numbers are $\emptyset$-definable in
$\mathbb{N}$ (see \cite{Computability}, \cite{Computability2}).

\begin{lemma} \label{bi-list}
The list superstructure $S(\mathbb{N,N})$ is $\emptyset$-interpretable in
$\mathbb{N}$. Moreover,  $S(\mathbb{N,N})$ and  $\mathbb{N}$ are bi-interpretable.\end{lemma}

We have shown in the previous section that $\mathbb{N}$ is
interpretable in a monoid $G$ satisfying the assumptions of Lemma \ref{lm:1}. Thus, by transitivity of
interpretability, we have that $S(\mathbb{N,N})$
is intepretable in $G$ and, therefore, in $\mathbb{M}_X$.

In this section we will construct a direct intepretation of
$S(\mathbb{N,N})$ in $\mathbb{M}_X$, which we will use in Section \ref{se:2.3}. The same technique also works for $A_{\Gamma}$ without center.

\begin{lemma}
 $S(\mathbb{N,N})$ is interpretable in
  $\mathbb{M}_X$ and in any non-commutative  $A_{\Gamma}.$
\end{lemma}
   
\begin{proof} We will give the proof for $\mathbb {M}_X$.
Recall that the set $C(x_1)=\{x_1^n 
\mid n \in \mathbb{N}\}$ is interpretable in $\mathbb{M}_X$ with
parameter $x_1$ and
similarly $C(x_2)$ is interpretable with parameter $x_2$. 

To interpret
$S(\mathbb{N},\mathbb{N})$ in $\mathbb{M}_X$, we first interpret a
tuple $t=(t_1,\ldots,t_m)$ in $S(N)$ with $m \geq 1$
as a word

\begin{equation}
 w_t=x_1x_2^{t_1+1}x_1^2x_2^{t_2+1} \cdots x_1^mx_2^{t_m+1}
\end{equation}

Note that any such word $w_t$ is completely determined by $t$ and the following
conditions:
\begin{enumerate} 
\item [1)] (head) $w_t=x_1x_2g_1$ 
\item [2)] (recursion) If $w_t=g_3x_1^ig_4$ where $g_3 \neq g_3'x_1$, $g_4=x_2g_4'$, then $g_4=g_5x_1^{i+1}g_6$ where $g_5 \in C(x_2)$ and
  $g_6=x_2g_6'$, or $g_4 \in C(x_2)$.
\item [3)] (tail) $w_t=g_7x_2$
\end{enumerate}


Conditions 1)--3) are definable in $\mathcal{L}_{\{x_1,x_2\}}$, so
there is a formula $w(x)$ defining the set of words $w_t$ for $t \in S(N)$.

       Next, we interpret the relations $\in$, $t(s,i,a)$,
       $length(s,n)$. The set of triples $(w,x_1^i,x_1^a)$, where $a$ is
 the $i^{th}$ component of the tuple given by $w$, can be defined
 by the following position formula, which says roughly that $x_1^ix_2^{a+1}$ is
 a subword of $w$:

\begin{align*}
t(x,y,z)\colon &w(x) \wedge y \in C(x_1) \wedge\\
& (\exists g_1,g_2,g_1',g_2',v \;\;
x=g_1yx_2vg_2  \wedge g_1 \neq g_1'x_1 \wedge g_2 \neq x_2g_2' 
\wedge\\
&v \in C(x_2) \wedge Trans(v,z))
\end{align*}

The set of pairs $(x_1^a,w)$ where $a$ is a component of the tuple
encoded in $w$, can be defined by the formula $ In(x,y): \,
\exists z \, t(y,z,x)$. Finally, the length relation  can be defined by
the formula $l(x,y): w(x) \wedge y \in C(x_1) \wedge \exists
g_1,g_2,g_1' \, x=g_1yg_2 \wedge
g_1 \neq g_1'x_1 \wedge g_2 \in C(x_2)$.

 
Next we interpret the concatenation operation in $\mathbb{M}_X$. Suppose we have words $w_1$ and $w_2$ corresponding to the tuples
   $(t_1,\ldots ,t_m)$ and $(p_1,\ldots ,p_n)$ respectively. Let
   $w_2'=x_1^{m+1}x_2^{p_1+1} \cdots x_1^{m+n}x_2^{p_{n+1}}.$ Then $w_2'$
   has the following properties:
\begin{enumerate}
\item [1)] (head)$w_2'=x_1^{m+1}x_2g_1$, where $m$ is the length of $w_1$
\item [2)](recursion) If $w_2'=g_2x_1^{m+i}g_3$ where $g_2 \neq g_2'x_1$ and $g_3=
x_2g_3'$, then $g_3=x_2^{p_i+1} x_1^{m+i+1}g_4$, where
  $g_4=x_2g_4'$, or $g_3=x_2^{p_i+1}$.
\item [3)] (tail) $w_2'=g_5x_2$
\end{enumerate}

All of these properties are definable with
parameters $x_1$, $x_2$, and with the formulas
defining the interpretations of the length and position functions. Thus, there is a formula $\phi
(x,y, z)$ such that $\mathbb{M}_X \models \phi(w_1,w_2,w_2')$ when $w_1,w_2,w_2'$
are as above. Now let $t_3$ be the concatenation of the tuples $t_1$
and $t_2$. Let the corresponding words be $w_1,w_2,w_3$
respectively. Then the formula $Concat(x,y,z)$: $\exists u \,
\phi(x,y,u) \wedge z=xu$ defines concatenation uniformly in $X$, that is,
$Concat(w_1,w_2,w_3)$ holds when $t_1 \frown t_2 =t_3$ and if instead of $x_1,x_2$ we take another basis $x_i,x_j$ we just have to replace in the formula  $x_1,x_2$ by $x_i,x_j$.

To eliminate parameters we can now, as in the proof of Theorem \ref{th:1}, define by a formula the equivalence relation identifying elements $w_t(x_i,x_j)$ for all pairs $(x_i,x_j)$ of  basis elements. 

The proof for $A_{\Gamma}$ is similar.
\end{proof}

\subsection{Interpretation of $\mathbb{M}_X$ and other monoids in
  $S(\mathbb{N},\mathbb{N})$ }\label{biint}

Let $X=\{x_1,\ldots , x_n\}$. We interpret a monomial
$x_{i_1}x_{i_2}\cdots x_{i_m} \in \mathbb{M}_X$ as the tuple
$(i_1,i_2,\ldots ,i_m)$. Let $T=\{(t_1,\ldots ,t_m) \mid 1 \leq t_i \leq n, \, m
\in \mathbb{N}\}$, then any element of $T$ can be uniquely associated to a
monomial in $\mathbb{M}_X$. So, $\mathbb{M}_X$ can be interpreted in
$S(\mathbb{N},\mathbb{N})$ as the set $T$. It is easy to see $T$ is
definable since the conditions $1 \leq t_i \leq n$ and $m \in
\mathbb{N}$ can be written in the language of
$S(\mathbb{N},\mathbb{N})$. Multiplication in $\mathbb{M}_X$ can be
interpreted as concatenation. So, $\mathbb{M}_X$ is $\emptyset$-interpretable in
$S(\mathbb{N},\mathbb{N})$.

\begin{lemma} \label{lm:nonf} If $G$ is a monoid from Theorem \ref{th:!}, a), b) or a monoid with solvable word problem from c), then $G$ is interpretable in $S(\N,\N)$ and in $\N$.\end{lemma} 
\begin{proof} Notice that the word problem is solvable in monoids from a) and b). One can recursively enumerate all short-lex  forms of elements in G and encode them as tuples in $\N$  the same way as this is done for $\mathbb{M}_X.$ Multiplication is not just concatenation anymore, but the corresponding predicate is recursively enumerable and therefore definable in $S(\N,\N)$.\end{proof}

\subsection{Bi-interpretation}\label{se:2.3}
\begin{theorem}\label{th:bi}
1. $S(\mathbb{N,N})$ and $\mathbb{M}_X$  are bi-interpretable with parameters in $X$ uniformly in $X$.

2. If a non-trivial free partially commutative monoid $A_{\Gamma}$ has trivial center, then $S(\mathbb{N,N})$ and $A_{\Gamma}$ are bi-interpretable with the standard generating set (vertices $V$ of $\Gamma$) as parameters.\end{theorem}

1. We first will prove bi-interpretability of $S(\mathbb{N,N})$ and $\mathbb{M}_X$   uniformly in $X$.

Denote by $\mathbb{M}_X^*$ the interpretation of $\mathbb{M}_X$ in   $S(\mathbb{N,N})$, and by
$S(\mathbb{N,N})^*$ the interpretation of $S(\mathbb{N,N})$ in $\mathbb{M}_X$.  Denote the images of $ \mathbb{M}_X$ and
$S(\mathbb{N,N})$ in themselves by $\mathbb{M}_X^{**}$ and
$S(\mathbb{N,N})^{**}$, respectively.  
To show bi-interpretability, it remains to show that the isomorphisms
  $S(\mathbb{N,N}) \to S(\mathbb{N,N})^{**}$ and $\mathbb{M}_X \to
 {\mathbb{M}_X}^{**}$ are definable in $S(\mathbb{N,N}) $ and
  $\mathbb{M}_X$, respectively. The isomorphism
  $\psi:S(\mathbb{N,N}) \to S(\mathbb{N,N})^{**}$ is the composition
  of the map taking a tuple $t=(t_1,\ldots,t_m) \mapsto
  w_t=x_1x_2^{t_1+1} \cdots x_1^mx_2^{t_m+1} $ and the map taking $M=
  x_{t_1} \cdots x_{t_m}  \mapsto t_M=(t_1,\ldots,t_m)$. So, $
\psi (t)=(1,2,\ldots,2) \frown (1,1,2,\ldots,2) \frown \cdots \frown
(1,\ldots,1,2,\ldots,2)$ where the $i^{th}$ tuple has $i$ $1$'s and $t_i+1$ $2$'s.
Since every recursively enumerable predicate is definable in $\N$ we have the following
\begin{lemma}
The isomorphism $\phi:S(\mathbb{N,N}) \to S(\mathbb{N,N})^{**} $ mapping
$t \mapsto t_{w_t}$ is $\emptyset$-definable in $S(\mathbb{N,N}).$ 
\end{lemma}

To show that the isomorphism $\phi: \mathbb{M}_X \to
 {\mathbb{M}_X}^{**}$ is definable, note that this map is the
 composition of the map sending $x_{i_1} \cdots x_{i_m} \mapsto
 (i_1,\ldots,i_m)$ and the map sending $(i_1,\ldots,i_m) \mapsto
 x_1x_2^{i_1+1} \cdots x_1^mx_2^{i_m+1}$. We will show that this
 isomorphism is definable with parameters in $X$. Recall that the set
 $X$ is definable in $\mathbb{M}_X$. 

  We first define a relation $R=\{(x_1,x_1),
  (x_2,x_1^2),\ldots,(x_n,x_1^n)\}$ that pairs up the index of an element
  in $X$ with its interpretation. In the language $\mathcal{L}_X$, this relation
  is certainly definable. We will call elements $(x_i,x_1^i)$ pairs.

Next, given a number $m \in \mathbb{N}$, we define an element $a_m \in
\mathbb{M}_{\{x_1,x_2\}}$ by 

\begin{equation} 
a_m=x_2x_1x_2x_1^2x_2x_1^3 \cdots x_2x_1^m
  \end{equation}

\begin{lemma}\label{lemma-a}
The set of pairs $B=\{(a_m,x_1^m) \mid m \in \mathbb{N}, m>0 \}$ is
definable in $\mathbb{M}_X$. 
\end{lemma}

\begin{proof}
The monomials $a_m$ are completely determined by $m$ and the following
conditions:

\begin{enumerate}
\item [1)] (head) $a_m=x_2x_1x_2u$
\item [2)] (recursion) If $a_m=u_1x_2vx_2u_2$ with $v \in C(x_1), u_2 \neq
  x_1^m$, then $u_2=vx_1x_2u_3$.
\item [3)]  (tail) $a_m=u_4x_2x_1^m$
\end{enumerate}

The conditions are definable in $\mathbb{M}_X$ in the language
$\mathcal{L}_{\{x_1,x_2\}}$,  so the relation $B$ is definable in
$\mathbb{M}_X$ with parameters $x_1$ and $x_2$.
\end{proof}

\begin{lemma} \label{iso-M}

The isomorphism $\phi: \, \mathbb{M}_X^{**} \to {\mathbb{M}_X}$
sending $w_M=x_1x_2^{i_1+1} \cdots
x_1^mx_2^{i_m+1} \mapsto M=x_{i_1} \cdots x_{i_m}$ is definable in $ \mathbb{M}_X$ with parameters in $X$ uniformly in $X$.

\end{lemma}

\begin{proof}
Recall that for a word $w_M=x_1x_2^{i_1+1} \cdots
x_1^mx_2^{i_m+1}$ we have defined a length relation and the position
relation $t(s,i,a)$. Thus, for $w_M$ let $a=a_m$, where $m$ is the
length of $w_M$, and define a word $w$ as
follows:

\begin{equation}
w=(ax_2a)x_{i_1}(a^2x_2^2a^2)x_{i_1}x_{i_2}(a^3x_2^3a^3)
\cdots x_{i_1}x_{i_2}\cdots x_{i_m}(a^{m+1}x_2^{m+1}a^{m+1})
\end{equation}

The word $w$ is completely determined by $w_M$ and the following
conditions:
 
\begin{enumerate}
\item [1)] (head) $w=(ax_2a)x_{i_1}(a^2x_2^2a^2)v_1$, where
  $x_{i_1}$ is the pair of the first component of $w_M$.
\item [2)] (recursion) For any $j\in \mathbb N , 0<j<m$, and for any $v_2,v_3,v_4\in M$, if $w=v_2 (\bar ax_2^j\bar a) v_3 (\bar aax_2^{j+1}\bar aa)v_4$,  where
 $\bar a \in C(a)$,
 $v_2, v_3$ do not end in $a$,  $v_3,v_4$ do not start with $a$,
\newline
then  $v_4=v_3x_{i_{j+1}} (\bar a a^2x_2^{j+2}\bar a a^2) v_5$,
where $v_5$ does not begin with $a$, and $x_{i_{j+1}}$ is the pair of the
$(j+1)'{st}$ component of $w_M$

\item [3)] (tail) $w=v_6(\bar ax_2^{m+1}\bar a)$, where   $\bar a \in C(a)$ and $v_6$ does not end with $a$. 
\end{enumerate}

Conditions 1)--3) are definable with parameters in $X$. Thus, there
is a formula $\theta_0 (x,z, X)$ such that $\theta_0 (w_M,w, X)$
holds in $\mathbb{M}_X$ whenever $w_M, w$ are as we defined them.
Then the formula $\theta_1 (x,y,X): \exists z, y  \, (\theta_0(x,z,X) \wedge
  z=u(\bar ax_2^m\bar a)y(\bar aax_2^{m+1}\bar aa) \wedge
\forall u'\ (u \neq u'a)) $ defines a pair $(w_M, M)$ with parameters in $X$. 
\end{proof}

The first statement of Theorem \ref{th:bi} is proved now.

To prove the second statement we need an analog of Lemma \ref{iso-M}  but we have to make such an element $a=a_m$ that does not commute with any generator and 
the element $w$ is uniquely defined by 1)--3) and $w_M$. If in Lemma \ref{lemma-a} we replace any occurrence of $x_2$ in $a_m$ by the product of all the  generators in $V$ except $x_1$, then  the proof of Lemma \ref{iso-M} will work. This proves the second statement of Theorem \ref{th:bi} .

\begin{cor}\label{bi_M-N}  $\mathbb{M}_X$  and $\mathbb N$ are bi-interpretable with parameters $X$  uniformly in $X$.
\end{cor}
There are immediate very important corollaries.
Similarly to the arithmetic, the first-order theory of  every structure $\MB$ bi-interpretable with $\N$  has the same expressive power as the weak second order theory of $\MB$. Namely, every statement about  $\MB$ that  can be expressed in the weak second order logic of $\MB$ can be expressed in the first-order logic. 

\begin{cor}   If $\MB$  and $\N$ are bi-interpretable, then $\MB$ and $S(\MB,\mathbb N)$ are bi-interpretable. 
\end{cor}
\begin{proof} Since $\MB$ and $\N$ are bi-interpretable  we have that
$S(\MB,\mathbb N)$ and $S(\mathbb N,\mathbb N)$ are bi-interpretable.    At the same time
$\mathbb N$ and $S(\mathbb N,\mathbb N)$ are 
bi-interpretable. Therefore  $\MB$  and $S(\MB,\mathbb N)$ are bi-interpretable. \end{proof}
\begin{cor}   $\mathbb{M}_X$  and $S(\mathbb{M}_X,\mathbb N)$ are bi-interpretable with parameters $X$ uniformly in $X$.
\end{cor}
This result shows that shows that structures bi-interpretable with arithmetic are rich (meaning that they have many definable sets).
\begin{cor}   A non-trivial $A_{\Gamma}$  with trivial center and $S(A_{\Gamma}, \mathbb N)$ are bi-interpretable with parameters $V$ uniformly in $V$.
\end{cor}

\section{Definability of a submonoid }

Consider now the submonoid of $\mathbb{M}_X$ generated by the elements $g_1, \ldots, g_k$, that is,
$\langle g_1,\ldots,g_k\rangle$.

\begin{theorem} \label{th:sub}
For any $k \in \mathbb{N}$, there is a formula $\psi(y,y_1,\ldots, y_k,X)$
such that $\psi(g,g_1,\ldots,g_k,X)$ holds in $\mathbb{M}_X$  if and only if
$g \in \langle g_1,\ldots,g_k\rangle$.  

Such a formula also exists for any non-trivial free partially commutative monoid 
$A_{\Gamma}$ with trivial  center. 
\end{theorem}

We will give a proof for $\mathbb{M}_X$. We will use the fact that the structures $\mathbb{N}$
and $S(\mathbb{N,N})$ are bi-interpretable with $\mathbb{M}_X$.

 A Diophantine
equation is an equation of the form $p(x_1,\ldots ,x_k)=0$, where $p(x_1,\ldots ,x_k)\in  \Z [x_1,\ldots ,x_k].$ A solution to the Diophantine equation is an assignment
$a:\{x_1,\ldots,x_k\} \to \mathbb{Z}$ such that $p(a(x_1),\ldots ,a(x_k))=0$ in
$\mathbb{Z}$. A set $K \subset \mathbb{Z}^k$ is said to be Diophantine
if there is a polynomial $p(x_1,\ldots,x_n,y_1,\ldots,y_k)$ such that for any
$(a_1,\ldots,a_k) \in \mathbb{N}^k$ the equation
$p(x_1,\ldots,x_n,a_1,\ldots,a_k)=0$ has a solution in $\mathbb{Z}$ if and
only if $(a_1,\ldots,a_k) \in K$. 
   
 A set $K \in \mathbb{Z}$ is recursive if there is a
 recursive function $f:K \to \{0,1\}$ such that $n \in K$ if and only
 if $f(n)=1$. A set $K$ is recursively enumerable if it is
 the range of a total recursive function. Matisyasevich \cite{M} proved that any recursively enumerable set is
Diophantine. 

In particular, he proved the following:

\begin{prop}
Every recursively enumerable relation $A(a_1,\ldots,a_m)\in \mathbb N^m$ can be
represented in the form 
\begin{align}
A(a_1,\ldots,a_m) \iff \exists (x_1,\ldots , x_n) \in \mathbb N^n \, P(a_1,\ldots,a_m,x_1,\ldots,x_n)=0,
\end{align}  where $P$ is a polynomial  with integer
coefficients.
\end{prop}

Recall from Lemma \ref{bi-list} that $S(\mathbb{N,N})$ is $\emptyset$-interpretable in $\mathbb{N}$. We will refer to the interpretation in $\mathbb{N}$ of a finite sequence (tuple) in $S(N)$ as its code.

{\em Proof of Theorem \ref{th:sub}.}
Consider now $W=\langle g_1,\ldots,g_k \rangle $ and recall that each $g_i=x_{i_1}
\cdots x_{i_m}$ has an interpretation in $S(\mathbb{N,N})$ as the tuple
$t_i=(i_1,\ldots,i_m)$, and this tuple in turn is interpreted as a code
$n_i \in \mathbb{N}$. The set of words in $\langle g_1,\ldots,g_k\rangle $ can be
recursively enumerated. Therefore the set of all tuples $(g_1,\ldots ,g_k,g)$ such that  $g\in \langle g_1,\ldots,g_k\rangle $ is also recursively enumerable.  
Therefore the set $W_k=\{(n_1,\ldots,n_k,s)\}$ of $k+1$-tuples of codes of $(g_1,\ldots ,g_k,g)$ in $\N$  is also recursively enumerable.

   By Matisyasevich's theorem in \cite{M}, we have that the set $W_k$ is Diophantine. So, there
is a polynomial $P(x_1,\ldots,x_n,n_1,\ldots,n_k,s)$ with integer
coefficients such that $P=0$ has a solution in $\mathbb{Z}$ if and
only if $(n_1,\ldots,n_k,s) \in W_k$. Thus, the formula
$\phi(y_1,\ldots,y_k,z): \, \exists x_1,\ldots,x_n \,
P(x_1,\ldots,x_n,y_1,\ldots,y_k,z)=0$ defines $W_k$ in $\mathbb{Z}$. Since
$\mathbb{N}$ is definable in $\mathbb{Z}$, there is some formula
$\phi'(y_1,\ldots,y_k,z)$ which defines $W_k$ in $\mathbb{N}$.

   To show that the set $S_k=\{(g_1,\ldots,g_k,g) \mid g \in
   \langle g_1,\ldots,g_k\rangle\}$ is definable in $\mathbb{M}_X$, we
   use the result in Lemma \ref{lemma:int}.

The formula $\phi'(y_1,\ldots,y_k,z)$ defines
the set $W_k=\{(n_1,\ldots,n_k,s)\}\in \N ^{k+1}$ where for each $1 \leq i \leq k$,
$n_i$ is the code of
an element $g_i \in \mathbb{M}_X$ and $s$ is the code of an element $g
\in \langle g_1,\ldots,g_k\rangle$.  Since $\mathbb{N}$ is $\emptyset$-interpretable in
$\mathbb{M}_X$, by Lemma \ref{lemma:int} there is a formula $\phi ^*(y_1,\ldots,y_k,z,X)$ in
$\mathbb{M}_X$ such that for any $n_1,\ldots,n_k,s \in \mathbb{N}$,
$\mathbb{N} \models \phi'(n_1,\ldots,n_k,s)$ if and only if $\mathbb{M}_X
\models \phi^* (n_1^*,\ldots,n_k^*,s^*,X)$, where $n_1^*,\ldots,n_k^*,s^*$ are the images of $n_1,\ldots,n_k,s$ in $\mathbb{M}_X$. By Lemma \ref{iso-M} the set of tuples
$\{(n_1^*,\ldots,n_k^*,s^*,  g_1,\ldots,g_k,g,\}$ is definable in
$\mathbb{M}_X$ by some  formula $\theta (n_1^*,\ldots,n_k^*,s^*,  g_1,\ldots,g_k,g,X).$  Let $$\psi (g_1,\ldots,g_k,g,X)=$$ $$\exists n_1^*,\ldots,n_k^*,s^* (\phi^*(n_1^*,\ldots,n_k^*,s^*,X) \wedge \theta (n_1^*,\ldots,n_k^*,s^*,  g_1,\ldots,g_k,g,X)).$$ Then $\mathbb{N} \models
\phi'(n_1,\ldots,n_k,s)$ if and only if  $\mathbb{M}_X\models \psi (g_1,\ldots,g_k,g,X)$ if and only if $g \in \langle g_1,\ldots,g_k\rangle $, so we have
our result.

Similarly one can prove the result for $A_{\Gamma}$ without the center and the following result.

\begin{theorem} Every recursively enumerable language in the alphabet $X$ is definable in $\mathbb{M}_X$.
\end{theorem}

This implies that every regular language is definable in $\mathbb{M}_X$.

\section{Isolation of Types, Homogeneity,  and QFA property}

Let $A \subset \MB$. A set $p$ of $\mathcal{L}_A$-formulas in $n$ free
variables is called an \textit{$n$-type} of $Th(\MB, \{a\}_{a \in A})$ if $p \cup Th_A(\MB)$ is
satisfiable. A type $p$ is called \textit{complete} if for each $\mathcal{L}_A$-formula
$\phi$ with $n$ free variables, either $\phi$ or $\neg \phi$ is in
$p$. Moreover, $p$ is said to be \textit{realized} in $\MB$ if there is some $\bar b \in \MB$
such that $\MB \models  \phi (\bar b)$ for all $\phi \in
p$. For a tuple $\bar b \in \MB$, the set $tp^{\MB}(\bar b /A) =
\{\phi (\bar x) \in \mathcal{L}_A \mid \MB \models \phi (\bar b) \}$ is
a complete $n$-type.

A complete $n$-type $p$ is isolated if there is a formula $\phi( \bar
x) \in p$ such that for all $\mathcal{L}$-formulas $\psi(\bar x), \psi(\bar x)
\in p$ if and only if $Th(\MB) \models (\phi(\bar x) \implies
\psi( \bar x))$. Moreover, $\MB$ is called \textit{atomic} over
$A$ if every type that is realized in $\MB$ is isolated.

\begin{remark}
Let $\MB$ be a countable structure. Then $\MB$ is
atomic if for any $\bar b \in \MB$, the orbit $Aut(\MB).\bar b$ is
$\emptyset$-definable.
\end{remark}
A model is {\em homogeneous} if  two finite tuples realize the same types if and only if they are automorphically equivalent.  Every countable atomic model is homogeneous.
\begin{theorem} $\mathbb{M}_X$ is atomic and, therefore, homogeneous.
\end{theorem}

\begin{proof}
Note that since the basis $X$ of $\mathbb{M}_X$ is definable, any
automorphism must send basis elements to basis
elements. Morever, an automorphism is completely determined by where it
sends the basis elements. Thus,
the orbit of a word $w=x_{i_1}^{e_1} \cdots x_{i_n}^{e_n} \in
\mathbb{M}_X$ is the set of words $\{s=x_{\sigma(i_1)}^{e_1} \cdots
x_{\sigma(i_n)}^{e_n}\}$ where $\sigma$ is a permutation of the set
$\{1,\ldots,n\}$. It is easy to see that this set is definable. For
example, the orbit of a word $x_1^2x_3x_2x_3$ can be defined by the
formula $\phi(x): \, \exists y_1,y_2,y_3 \in X \, (y_1 \neq y_2 \neq
y_3  \, \wedge \, x = y_1^2y_2y_3y_2)$. 
Similarly, we can show that the orbits of arbitrary tuples $\bar b$ of
$\mathbb{M}_X$ are definable. Thus, $\mathbb{M}_X$ is atomic.

\end{proof}
The same result is true for $A_{\Gamma}$ without the center because the standard generating set is definable. Notice that non-abelian free groups are also homogeneous \cite{PS}, \cite{O}.

\begin{definition} Fix a finite signature. An infinite finitely generated structure $G$ is
quasi-finitely axiomatizable (QFA) if there is a first-order sentence
$\phi$ such that \begin{itemize}
\item $G \models \phi $,
\item if $H$ is a finitely generated structure in the same signature such that $H\models \phi$,
then $G=H.$\end{itemize}\end{definition}

A structure
$G$ is {\em prime} if $G$ is isomorphic to an elementary submodel of each $H$ elementarily equivalent to $G$. 

\begin{remark} If $|X|>1$, then $\mathbb{M}_X$ is QFA and prime. A non-trivial $A_{\Gamma}$ with trivial  center  is QFA and prime.
\end{remark} 
This follows from \cite{nies}, Theorem 7.14 and Corollary \ref{bi_M-N}. Indeed, by \cite{nies} a finitely generated structure $A$ in  finite signature that is bi-interpretable with integers is prime and QFA.

In contrast to this two non-abelian free groups of different ranks are elementarily equivalent, therefore a non-abelian free group is not QFA.
\section{Quantifier elimination}\label{elimination}

In this section we will show that there is no quantifier elimination in the theory of any structure that is bi-interpretable with $\mathbb N$.  In particular, there is no quantifier elimination in the theory of a free monoid of rank at least two.

Let $\mathcal{L}$ be a first-order language. Recall that a  formula $\phi$ in $\mathcal{L}$ is in a prenex normal form if $\phi = Q_1 y_1Q_2y_2 \ldots Q_sy_s\phi_0(x_1, \ldots,x_m)$ where $Q_i$ are quantifiers ($\forall$ or $\exists$), and $\phi_0$ is a quantifier-free formula in $\mathcal{L}$. It is known that every formula in $\mathcal{L}$ is equivalent to a formula in the prenex normal form.
A formula $\phi = Q_1 y_1Q_2y_2 \ldots Q_sy_s\phi_0(x_1, \ldots,x_m)$ in the prenex normal form is called $\Sigma_n$ formula if the sequence of quantifiers  $ Q_1 Q_2 \ldots Q_s$ begins with the  existential quantifier $\exists$  and alternates  $n-1$  times between series of existential and universal quantifiers. Similarly,  a formula $\phi$ above is $\Pi_n$ formula if the sequence of quantifiers  $ Q_1 Q_2 \ldots Q_s$ begins with the  universal  quantifier $\forall$  and alternates $n-1$  times between series of existential and universal quantifiers. 

For a structure  $\MB$  of the language $\mathcal{L}$ denote by $\Sigma_n(\MB)$ the  set of all subsets of ${\MB}^m$, $m \in \N$, definable in $\MB$ by  $\Sigma_n$ formulas  $\phi(x_1, \ldots,x_m)$, $m \in \N$. Replacing in the definition above  $\Sigma_n$ by $\Pi_n$  one gets the set $\Pi_n(\MB)$.  Let  $\Sigma_0(\MB) = \Pi_0(\MB)$ be the set of all subsets definable in $\MB$ by quantifier-free formulas.  Clearly,
$$
\Sigma_0(\MB) \subseteq  \Sigma_0(\MB) \subseteq \ldots \subseteq \Sigma_n(\MB) \subseteq \ldots
$$
$$
\Pi_0(\MB) \subseteq  \Pi_0(\MB) \subseteq \ldots \subseteq \Pi_n(\MB) \subseteq \ldots
$$
 The sets $\Sigma_n(\MB)$ and $\Pi_n(\MB)$ form the so-called {\em arithmetical hierarchy} over $\MB$ denoted by ${\mathcal H}(\MB)$.  It is easy to see that if $\Sigma_n (\MB) = \Sigma_{n+1}(\MB)$ (or $\Pi_n (\MB) = \Pi_{n+1}(\MB)$) for some $n \in \N$ then $\Sigma_m (\MB) = \Sigma_{m+1}(\MB)$ and $\Pi_m (\MB) = \Pi_{m+1}(\MB)$ for every natural $m \geq n$.  We say that the hierarchy ${\mathcal H}(\MB)$ {\em collapses} if $\Sigma_n (\MB) = \Sigma_{n+1}(\MB)$ for some $n \in \N$, otherwise it is called {\em proper}.

\begin{theorem}
Let $\MB$ be a structure in the language $\mathcal{L}$ that is bi-interpretable with $\N$. Then for any $n \in \N$ there is a formula $\phi_n$ in $\mathcal{L}$ such that the formula $\phi_n$ is not equivalent in $\MB$ to any boolean combination of formulas from $\Pi_n$ or $\Sigma_n$ (with constants from $\MB$).
\end{theorem}
\begin{proof}  

Suppose, to the contrary, that for some $n \in \N$ any formula  $\phi (\bar x)$ in the language $\mathcal{L}$ is equivalent in $\MB$ to some boolean combination  $\phi'(\bar x)$ of formulas from $\Pi_n$ or $\Sigma_n$ with constants from $\MB$. Take  an arbitrary first-order formula $\psi(\bar z)$ of the language of $\N$. Since $\MB$ is bi-interpretable in $\N$ the  formula $\psi(\bar z)$ can be rewritten into  a formula $\phi (\bar x)$ of the  language $\mathcal{L}$ such that for any values $\bar a$ of $\bar z$, $\N \models \psi(\bar a) \Longleftrightarrow S \models \phi (\bar b)$, where $\bar a\rightarrow\bar b$ when $\N$ is interpreted in $S$. By our assumption there is a formula $\phi'(\bar x)$, which is a boolean combination   of formulas from $\Pi_n$ or $\Sigma_n$ perhaps with constants from $\MB$ such that $\phi(\bar x)$ is equivalent to $\phi'(\bar x)$ in $\MB$. Since $\MB$ is bi-interpretable in $\N$  there is a number $m$ which depends only on the bi-interpretation such that $\phi'(\bar x)$ can be rewritten into a formula $\psi'(\bar z)$, which is a boolean combination of formulas from $\Pi_{n+m}$ or $\Sigma_{n+m}$  in the language of $\N$ such that $\N \models \psi'(\bar a) \Longleftrightarrow S \models \phi'(\bar b)$. It follows that  $\psi (\bar z)$ is equivalent to $\psi '(\bar z)$ in $\N$, i.e., every formula $\psi$ of the language of $\N$ is equivalent in $\N$ to some formula $\psi '$ which is a boolean combination of formulas from $\Pi_{n+m}$ in the language of $\N$. However, this is false since the arithmetical hierarchy in $\N$ is proper. It follows that our assumption is false, so the theorem holds.
\end{proof}

\begin{cor}
The hierarchy ${\mathcal H}(\mathbb{M}_X)$ is proper.
\end{cor}

We would like to thank the referee for carefully reading our manuscript
and for giving constructive comments which substantially helped improving the
quality of the paper and inspired new results.  We would like to thank Alexei Miasnikov and Ben Steinberg for useful comments and discussions.

\end{document}